\documentclass[11pt,reqno,twoside]{amsart}
\usepackage{latexsym}
\usepackage[figuresright]{rotating}
\makeatletter
\def\LaTeX{\leavevmode L\raise.42ex
\hbox{\kern-.3em\size{\sf@size}{0pt}\selectfont A}\kern-.15em\TeX}
\makeatother

\newcommand{\BibTeX}{{\rm B\kern-.05em{\sc i\kern-.025emb}\kern-.08em\TeX}}

\usepackage{a4}
\usepackage[only,ninrm,elvrm,twlrm,sixrm,egtrm,tenrm,sevrm,fivrm]{rawfonts}
\usepackage{amssymb,amsfonts,amsmath}

\newtheorem{thm}{Theorem}[section]

\newtheorem{lem}[thm]{Lemma}
\newtheorem{prop}[thm]{Proposition}
\newtheorem{cor}[thm]{Corollary}

\theoremstyle{definition}

\newtheorem{defn}[thm]{Definition}

\newtheorem{exmp}[thm]{Example}

\makeatletter
\label{e:dispaa}
\label{e:dispau}
\label{e:dispav}
\label{e:dispaw}
\label{e:dispax}
\def\theequation{\thesection.\@arabic\c@equation}
\makeatother
\makeatletter

\usepackage{color}

\newcommand{\benu}{\begin{enumerate}}
\newcommand{\enu}{\end{enumerate}}

\usepackage{amssymb,amsfonts,amsmath}
\usepackage[all,dvips]{xy}

\makeatother

\begin{document}
\title[Fundamental group]
{The fundamental group of an algebra with a strongly simply connected Galois covering}

\author[Castonguay]{Diane Castonguay}
\address{Instituto de Inform\'atica Bloco IMF I, Campus II, Samambaia,
Universidade Federal de Goi\'as, Goi\~ania, GO, Brazil.}
\email{diane@inf.ufg.br}
\author[Chaio]{Claudia Chaio}
\address{Centro Marplatense de Investigaciones  Matem\'aticas, Facultad de Ciencias Exactas y Naturales,
Funes 3350, Universidad Nacional de Mar del Plata, 7600 Mar del
Plata, Argentina} \email{claudia.chaio@gmail.com}
\author[Trepode]{Sonia Trepode}
\address{Centro Marplatense de Investigaciones Matem\'aticas, Facultad de Ciencias Exactas y Naturales,
Funes 3350, Universidad Nacional de Mar del Plata, 7600 Mar del
Plata, Argentina} \email{strepode@mdp.edu.ar}

\keywords{presentation, fundamental group, Galois
covering, strongly simply connected.}

\subjclass{16G70, 16G20, 16E10}
\maketitle

\begin{abstract}
In this work, we prove that if a triangular algebra $A$ admits a strongly simply
connected universal Galois covering for a given presentation then the fundamental
group associated to this presentation is free.
\end{abstract}
\vspace{.3 cm}

\section*{Introduction}

Let $A$ be a finite dimensional $k$-algebra over an algebraically closed field,  $k$. We denote by $\mbox{mod}\,A$  the module category
whose objects are the finitely generated left $A$-modules.

Covering techniques had played an important role in the representation theory of algebras.
In~\cite{BG}, K. Bongartz and P. Gabriel introduced the concept of simply connected algebras with the aim of studying representation-finite algebras
using  covering techniques. Moreover, P. Gabriel reduced the problem of studying the representation theory of representation-finite algebras to the problem of studying the representation theory of its simply connected Galois covering.

In the representation-infinite case covering techniques are more difficult to handle. The case of polynomial growth algebras  have been  considered by A. Skowro\'nski in \cite{AS:93}. On the other hand,  selfinjective algebras have also been studied having into account covering techniques.

In this paper we shall consider triangular algebras, that is, algebras $ A=kQ_A/I_A$
such that $Q_A$ has no oriented cycles. More precisely, we are concerned with triangular algebras which admit a strongly simply connected Galois covering for a given presentation $kQ_A/I_A$ of $A$. The main objective of this paper is to study their fundamental group. When the group  $G$ acting on the Galois covering $A$ is free some nice properties hold.

In  \cite{CP}, D. Castonguay and J. A. de la Pe\~na  introduced a class of algebras called algebras of the  first kind respect to
a given presentation $kQ_A/I_A$  of $A$. They proved  that for such algebras the fundamental
group $\pi_1(Q_A,I_A)$ is free.
In order to get our main result we prove that if $A$ is a triangular algebra which admits a strongly simply connected
Galois covering for a given presentation $kQ_A/I_A$ then $A$ is
of the first kind respect to
such presentation.  As a consequence we get that if  there is a presentation of a triangular algebra which admits a strongly simply connected
Galois covering then the fundamental group  is free.  More precisely,
we prove Theorem A.
\vspace{.1in}

\noindent{\bf Theorem A.} {\it Let $A$ be a triangular algebra which admits a strongly simply connected
Galois covering for a given presentation $kQ_A/I_A$ of $A$. Then,
the fundamental group $\pi_1(Q_A, I_A)$ is free.}

\vspace{.1in}

This text is organized as follows. In section 1, we recall some preliminary results
and we present two results due to
I. Assem, S. Liu and Y. Zhang, (private communication).
Section 2 is devoted to prove Theorem A.

\section{Preliminaries}

\subsection{} {\bf Locally bounded categories.}  Let $k$ be an algebraically closed field.
A \textit{$k$-category} $A$ is a category where for each pair of objects $x,y$ of $A$, the
set of morphisms $A(x,y)$ has a $k$-vector space structure and where the
composition of morphisms is $k$-bilinear.

We denote by $A_0$ the set of objects of the $k$-category $A$.\vspace{.05in}

A $k$-category $A$ is called \textit{locally bounded} if it satisfies the following conditions.
\begin{enumerate}

\item for each $x \in A_0$, the endomorphism algebra $A(x,x)$ is local;

\item distinct objects in $A_0$ are not isomorphic;

\item for each $x \in A_0$, $\sum_{y \in A_0} \mbox{dim}_k A(x,y) <\infty$ and $\sum_{y \in A_0}
\mbox{dim}_k A(y,x) <\infty.$
\end{enumerate}

\vspace{.05in}

A {\it quiver\;} $Q$ is given by two sets $Q_0$ and
$Q_1$ together with two maps $s,e \colon Q_1 \rightarrow Q_0$. The elements of $Q_0$ are
called {\it vertices} and the elements of $Q_1$ are called {\it arrows}. We assign to
each arrow $\alpha \in Q_1$ its {\it start point} $s(\alpha)$ and its {\it end point}
$e(\alpha)$.  We say that $Q$  is  a \textit{locally finite quiver} if for any vertex $a \in Q_0$
there is a finite number of neighbours of a or equivalently, if there is a finite number of arrows
ending at $a$ and a finite number of arrows starting at $a$.\vspace{.05in}

If $A$ is a locally bounded category then there exists a locally
finite quiver $Q_A$ and an admissible ideal $I$ of the path
category $kQ_A$ of $Q_A$ such that there is an isomorphism $A
\simeq kQ_A/I$, called a presentation of $A$. The pair $(Q_A,I)$
is called a {\it bound quiver}.

The category $A$ is called {\it triangular} if the corresponding quiver $Q_A$ does not have oriented cycles.

For a basic background on representation theory we refer the reader to \cite{ARS, ASS}.

\subsection{Contours and cycles.}
We recall some notions from \cite{AL}. Let $Q$ be a locally finite
quiver. For  each vertex $a$ of $Q$, we denote by $\varepsilon_a$ the
trivial path at $a$. For each arrow $\alpha$ of $Q$, we
introduce a new arrow $\alpha^{-1}$ as its formal inverse with
$s(\alpha^{-1})=e(\alpha)$ and $e(\alpha^{-1})=s(\alpha)$.

We shall write the sequences of arrows from left to right.

We recall
that a {\it walk} in $Q$ is either a trivial path or a sequence
$w=c_1\cdots c_r$ with $c_i$ an arrow or the inverse of an arrow. When $s(w)=s(c_1)$ and $e(w)=e(c_r)$ we say that $w$ is a
walk from $s(w)$ to $e(w)$.

A non-trivial walk $w=c_1\cdots c_n$ with $s(w)=e(w)=a$ is called
a {\it cycle} at a point $a$. Observe that it could be a  non oriented cycle. We say that $w$ is {\it reduced} if
$c_{i+1} \ne c_i^{-1}$ for all $1\le i<n$ and $c_1\ne c_n^{-1}$;
and that $w$ is {\it simple} if $s(c_i)$ and $s(c_j)$  are pairwise
distinct points, where $1\le i,j\le n$.
\vspace{.1in}

By a path $p$ we mean an oriented walk, that is, when for all $i=1, \dots, n$ in $p$  we have
that $p =c_1\cdots c_n$ where $c_i$ are  arrows with $e(c_i)=s(c_{i+1})$  or they are all inverse of arrows with $s(c_i)=e(c_{i+1})$ and
in this case we write $p^{-1}=c_r^{-1}\cdots
c_1^{-1}$. Moreover, in the first case, if
$s(p)=s(c_1)$ and $e(p)=e(c_r)$ we say that $p$ is a path
walk from $s(w)$ to $e(w)$.
\vspace{.1in}

We say that the length of a path $p$ is $n$ if $p$ is a sequence of $n$ arrows or $n$ inverse of arrows, that is, $p= c_1\cdots c_n$ or $p^{-1}=c_n^{-1}\cdots
c_1^{-1}$.  We denote it by $\ell(p)$.  A non-trivial path $p=c_1\cdots c_n$ with $s(p)=e(p)$ is called
an {\it oriented cycle}.
\vspace{.1in}

A {\it contour} in $Q$ from a vertex $a$ to a vertex $b$ is a pair
$(p, q)$ of two non-trivial paths from $a$ to $b$.  We say that
a contour $(p, q)$ is {\it interlaced} if $p=p_1p_2$ and $q=q_1q_2$ with
$(p_1, q_1)$ and $(p_2, q_2)$ some contours, or equivalently if the paths $p$ and
$q$ have a common point other than $a$ and $b$;
and it is called {\it reducible}  if there
exist paths $p=p_0,  p_1, \ldots,  p_n=q$ in $Q$ from $a$ to $b$  such that, for each
$0 \le i <n$, the contour $(p_{i-1}, p_{i})$ is interlaced. In this case, we say that $p$
is reducible to $q$. Otherwise, it is called irreducible.
\vspace{.1in}

Next  we show a contour $(p,q)$ which is reducible and another which is irreducible.  \vspace{.05in}

\begin{exmp} $(a)$ Let $(p,q)$ be a contour in $Q$ from $a$ to $b$ as follows

$$
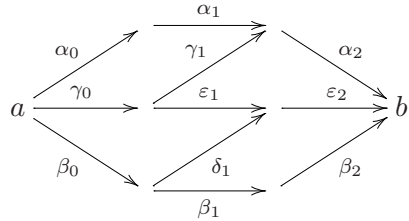
\xymatrix @!0 @R=1.1cm  @C=1.7cm  {     &\ar[r]^{\alpha_1}&\ar[rd]^{\alpha_2} \\
   a\ar[ru]^{\alpha_0} \ar[r]^{\gamma_0}  \ar[rd]_{\beta_0}&\ar[ru]^{\gamma_1} \ar[r]^{\varepsilon_1}& \ar[r]^{\varepsilon_2}& b   \\
  & \ar[ru]_{\delta_1}\ar[r]_{\beta_1} &  \ar[ru]_{\beta_2}\\
     }
$$

\noindent with $p= p_0={\alpha_0}{\alpha_1}{\alpha_2}$ and $q=p_4= {\beta_0}{\beta_1}{\beta_2}$. Consider the paths in  $Q$ from $a$ to $b$, $p_1= \gamma_0\gamma_1\alpha_2$,  $p_2= \gamma_0\varepsilon_1\varepsilon_2$ and $p_3= \beta_0 \delta_1 \varepsilon_2$.  Then, $p$
is reducible to $q$, since there is a sequences of paths $p=p_0, p_1, p_2, p_3, p_4=q$ where $(p_{i-1}, p_{i})$ are interlaced for $i=0, \dots, 4$.

$(b)$ Let $(p,q)$ be a contour in $Q$ from $a$ to $b$  as follows

$$
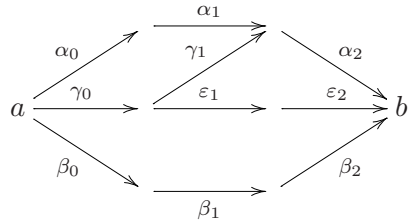
\xymatrix @!0 @R=1.1cm  @C=1.7cm  {     &\ar[r]^{\alpha_1}&\ar[rd]^{\alpha_2} \\
   a\ar[ru]^{\alpha_0} \ar[r]^{\gamma_0}  \ar[rd]_{\beta_0}&\ar[ru]^{\gamma_1} \ar[r]^{\varepsilon_1}& \ar[r]^{\varepsilon_2}& b   \\
  & \ar[r]_{\beta_1} &  \ar[ru]_{\beta_2}\\
     }
$$

\noindent with $p= p_0={\alpha_0}{\alpha_1}{\alpha_2}$ and $q=p_4= {\beta_0}{\beta_1}{\beta_2}$. Then, $(p, q)$
is irreducible.
\end{exmp}

Let $C=c_1\cdots c_n$ be a reduced cycle. We write $c_0=c_n$ and
$c_{n+1}=c_1$. By a source in $C$ we mean a vertex $a$ having two arrows starting in it.
We define $\sigma(C)$ to be the number of sources of
$C$. Then  $C$ is an oriented cycle if and only if $\sigma(C)=0$;
and $\sigma(C)=1$ if and only if there exists a unique $1\le i\le
n$ such that $c_{i+1}\cdots c_n c_1\cdots c_i=pq^{-1}$ with $p$ and $q$
some non-trivial paths. In this case, we say that $C$ forms the
contour $(p, q)$.\vspace{.05in}

Following \cite{AL}, we have the following definition of an irreducible cycle.

\begin{defn} Let $C$ be a simple cycle of $Q$.\begin{enumerate}

\item If $\sigma(C)=1$, then we say that $C$ is {\it irreducible} if
it forms an irreducible contour.

\item If $\sigma(C)> 1$, then $C$
is {\it reducible} provided that there exists a path $p$ from $a$ to $b$ ($p: a \rightsquigarrow b$) with $a$
and $b$ vertices of $C$, and two reduced walks $w_1$ and $w_2$ from $a$ to $b$ such that $C=w_1w_2$,
$\sigma(w_1p^{-1})< n$ and $\sigma(pw_2^{-1})<n$. In any other case, it is said to be \textit{irreducible}.
\end{enumerate} \end{defn}

Next, we show an example of an reducible cycle.

\begin{exmp} Assume that $Q$ is the following quiver:

$$
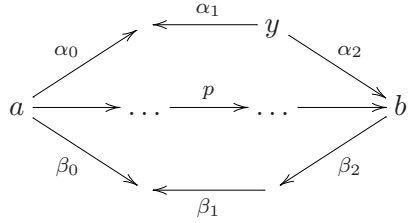
\xymatrix @!0 @R=1.1cm  @C=1.7cm  {     & &y \ar[l]_{\alpha_1} \ar[rd]^{\alpha_2} \\
   a \ar[ru]^{\alpha_0}\ar[rd]_{\beta_0}\ar[r] &\dots \ar[r]^{p}& \dots \ar[r] & \ar[ld]^{\beta_2} b   \\
  && \ar[l]^{\beta_1} \\
   }
$$

\noindent where $p$ is a path in $Q$ from $a$ to $b$. Consider the simple cycle $C= w_1 w_2$ with $w_1={\alpha_0}{\alpha_1}^{-1}{\alpha_2} $ and $w_2={\beta_0}{\beta_1^ {-1}}{\beta_2}^{-1}$. Note that $w_1$ and $w_2$ are reduced walks.
Moreover, observe that $C$ is not a contour, and that $x$ and $y$ are the only sources in $C$. Then $\sigma(C)> 1$.
By Statement $(2)$ of the above definition, we know that $C$ is reducible, since $\sigma(w_1p^{-1})< 2$ and $\sigma(p w_2^{-1})<2$.
\end{exmp}

\subsection{Homotopy} Let $(Q, I)$ be a connected locally finite
bound quiver. An element $\rho=\sum_{i=1}^n\lambda_ip_i\in I(a, b),$ with
$\lambda_i\in k$ and $p_i$ paths from  $a$ to $b$ of length at
least two, is called a {\it minimal relation} of $(Q, I)$ if $n\ge
2$ and for any proper subset $J$ of $\{1, 2, \ldots, n\}$ we have that
$\sum_{j\in J} \lambda_jp_j\not\in I(a, b)$. \vspace{.05in}

The {\it homotopy relation} $\sim$ on the set of walks of $Q$
is the smallest equivalence relation satisfying the
following conditions:\begin{enumerate}

\item[(a)] If $\alpha: a\to b$ is an arrow of $Q$ then
$\alpha\alpha^{-1}\sim \varepsilon_a$ and $\alpha^{-1}\alpha\sim
\varepsilon_b$.

\item[(b)] If $\sum_{i=1}^n\lambda_ip_i$ is a minimal relation then
$p_i\sim p_j$ for all $1\le i, j\le n$.

\item[(c)] If $u\sim v$ then $w_1uw_2\sim w_1vw_2$ for all walks $w_1,
w_2$ with $e(w_1)=s(u)$ and $s(w_2)=e(v)$. \end{enumerate}

We say that a cycle is {\it contractible} if it is
homotopic to a trivial path.\vspace{.05in}

Let $x\in Q_0$ be arbitrary. The set $\pi _1(Q,I,x)$ of
equivalence classes $\overline{u}$ of closed paths $u$ starting
and ending at $x$ has a group structure defined by the operation
$\overline{u}.\overline{v}= \overline{u.v}$. Since $Q$ is
connected then this group does not depend on the choice of $x$. We
denote it by $\pi _1(Q,I)$ and we call it the \textit {fundamental
group} of $(Q,I)$.\vspace{.05in}

We write the composition from left to right.

\begin{defn}
We say that a contour $(p, q)$ is a {\it torsion contour} in $(Q, I)$
if there is a positive integer $n>1$ such that
$(p q^{-1})^{n}$ is contractible. We say that
the order of the contour $(p, q)$  is $n$ if it is the minimal positive integer $n>1$ such that
$(p q^{-1})^{n}$ is contractible.
\end{defn}

\subsection{Natural homotopy} Let $(Q, I)$ be a locally finite
bound quiver. Two paths $p, q$ of $(Q, I)$ from a point $a$ to a
point $b$ are called {\it naturally homotopic} if $p=q$ or
if there exists a sequence of paths $p=p_0, p_1, \ldots,
p_n=q$ from $a$ to $b$ such that for each $1\le i<n$, we can write
$p_i=u_iv_{i1}w_i, p_{i+1}=u_iv_{i2}w_i$, where $u_i, w_i$ are
paths and $v_{i1}, v_{i2}$ are non-trivial paths appearing in the
same minimal relation of $(Q, I)$.  It follows immediately that
the natural homotopy is an equivalence relation on the set of the
paths in $(Q, I)$ which is compatible with the compositions of paths.
Moreover, two naturally homotopic paths are homotopic.
The converse is not true,
there are  homotopic paths which are not naturally homotopic as we show in the next example.

\begin{exmp} Consider the path algebra given by the quiver $Q$

$$
\xymatrix @!0 @R=0.1cm  @C=1.7cm  {     \ar[r]^{\alpha}& &  \\
 & \ar[r]^{\gamma} & \\
     \ar[r]_{\beta} & &  \\
     }
$$

\noindent with the ideal $I=< \alpha\gamma-\beta\gamma>$. We observe that $\alpha$ and $\beta$ are homotopic in $(Q, I)$ but they are
not naturally  homotopic. In fact, $\alpha\gamma \sim \beta\gamma$ then $\alpha \sim \beta$.

On the other hand, $\alpha$ and $\beta$ are not naturally  homotopic, since they do not have subpaths that form part of a minimal relation.
\end{exmp}

\begin{defn}[Assem, Liu, Zhang] Let $C=c_1\cdots c_r$ be a reduced cycle of $Q$ with
$\sigma(C)>0$.\begin{enumerate}

\item If $\sigma(C)=1$, then $C$ is called {\it naturally
contractible} if it forms a contour $(p, q)$ with $p, q$ naturally
homotopic.

\item If $\sigma(C)>1$, then $C$ is called {\it naturally contractible} if
there exists a path (it can be trivial)
$p: a \rightsquigarrow b$ with $a$
and $b$ vertices of $C$, and two reduced walks $w_1$ and $w_2$ from $a$ to $b$ such that $C=w_1w_2^{-1}$,
$\sigma(w_1p^{-1})< n$ and $\sigma(w_2p^{-1})<n$. Moreover both cycles $w_1p^{-1}$ and $w_2p^{-1}$
are naturally contractible non-oriented cycles.
\end{enumerate}
\end{defn}

Next, we show some examples comparing the concept of irreducible  or/and contractible cycles.

\begin{exmp} $(a)$. We present  an example of a reducible contour which is not a contractible cycle. Consider the following quiver

$$
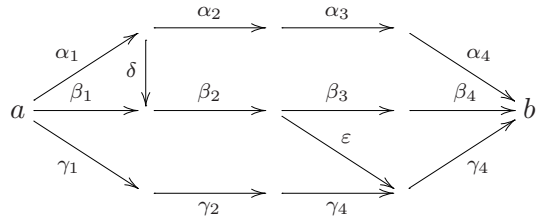
\xymatrix @!0 @R=1.1cm  @C=1.7cm  {     &\ar[r]^{\alpha_2}\ar[d]_{\delta}& \ar[r]^{\alpha_3} &\ar[rd]^{\alpha_4} & \\
   a\ar[ru]^{\alpha_1} \ar[rd]_{\gamma_1} \ar[r]^{\beta_1} & \ar[r]^{\beta_2}  &\ar[rd]^{\varepsilon} \ar[r]^{\beta_3}& \ar[r]^{\beta_4}& b   \\
  & \ar[r]_{\gamma_2} & \ar[r]_{\gamma_4} & \ar[ru]_{\gamma_4}&\\
     }
$$
\noindent with the relations $\alpha_2\alpha_3 =\delta \beta_2\beta_3\beta_4$, $\beta_1\beta_2\varepsilon =\gamma_1\gamma_2\gamma_3$ and $\beta_3\beta_4= \varepsilon\gamma_4$. Let $\alpha= \alpha_1\alpha_2\alpha_3\alpha_4$ and $\gamma=\gamma_1\gamma_2\gamma_3\gamma_4$. We claim that $(\alpha, \gamma)$ is a reducible contour since it is interlaced. Observe that $(\alpha, \gamma)$ is not a contractible cycle since $\beta_1 \not \sim \alpha_1\delta$. Therefore,  $\alpha \not \sim\gamma$.

$(b)$. Next, we show an example of an irreducible contour  which is not a contractible cycle.
Consider the contour

$$
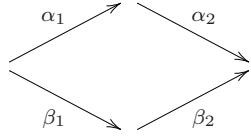
\xymatrix @!0 @R=0.9cm  @C=1.7cm  {&\ar[rd]^{\alpha_2} & \\
   \ar[ru]^{\alpha_1} \ar[rd]_{\beta_1}  & &  \\
  & \ar[ru]_{\beta_2} &
    }$$

\noindent without relations. Then $(\alpha, \beta)$ is irreducible but not a contractible cycle  since $\alpha \not \sim \beta$ where $\alpha=\alpha_1\alpha_2$ and $\beta=\beta_1\beta_2$.

$(c)$. Let $Q=(\alpha, \beta)$ with $\alpha=\alpha_1\alpha_2$ and $\beta=\beta_1\beta_2$ be the contour in $(Q,I)$ as follows

$$
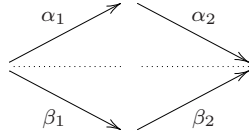
\xymatrix @!0 @R=0.9cm  @C=1.7cm  {&\ar[rd]^{\alpha_2} & \\
  \ar@{..}[r] \ar[ru]^{\alpha_1} \ar[rd]_{\beta_1}  & \ar@{..}[r]  & \\
  & \ar[ru]_{\beta_2} &
    }$$

\noindent where $I$ is the ideal generated by the relation $\alpha_1\alpha_2 = \beta_1\beta_2$. The given contour $Q$ is  irreducible and moreover is a contractible cycle because $\alpha_1\alpha_2 \sim \beta_1\beta_2$.
\end{exmp}

The following result is a direct consequence of the definition.

\begin{lem} \label{Lemma 1.6} Let $C=c_1\cdots c_r$ be a reduced cycle of $Q$.
Then \begin{enumerate}

\item If $C$ is naturally contractible in $(Q, I)$ then $C$ is
contractible and reducible in $(Q, I)$.

\item If there exists $1\le i<j<r$ such that $c_i\cdots c_j$ and
$c_{j+1}\cdots c_sc_1\cdots c_{i-1}\,$ with $c_0=c_s$ and
$c_{r+1}=c_1$ are naturally contractible cycles then $C$ is
naturally contractible.\end{enumerate} \end{lem}

\subsection{Strong simply connectedness}
Let $(Q', I')$ and $(Q, I)$ be locally finite bound quivers.  We
say that $(Q', I')$ is a {\it convex bound subquiver} of $(Q, I)$
if for every points $a, b$ of $Q'$, all paths from $a$ to
$b$ in $Q$ lie in $Q'$ and $I'(a, b)=(kQ')\cap I(a, b)$. We observe that the above convex property does not depends on the presentation.\vspace{.05in}

Let $A$ be a locally bounded category. A full subcategory $B$ of $A$ is called
{ \it convex} if, for any path $x_0 \rightarrow x_1 \rightarrow \dots \rightarrow x_t$  in (the quiver of) $A$ with $x_0, x_t \in B_0$ we have that
$x_i \in B_0$,  for all $1 \leq i < t$.
The category $A$ is called { \it triangular}  if its quiver $Q_A$ contains no oriented cycle.\vspace{.05in}

A triangular locally bounded $k$-category $R \simeq kQ/I$
is said to be {\it simply connected} if for every finite convex subcategory
$C$ of $R$ and any presentation $C\simeq kQ'/I'$ we have that the fundamental group of $(Q',I')$ is trivial, see \cite{BLS, BG} and  \cite{MP}.

\begin{defn} \label{ss} A triangular locally bounded $k$-category $R \simeq kQ/I$
is said to be strongly simply connected if the following two
conditions are satisfied:
\begin{enumerate}

\item For every  two vertices $x$ and $y$ in $Q$ there are only finitely
many paths in $Q$ from $x$ to $y.$

\item Every finite convex subcategory $C$ of $R$ is simply connected.

\end{enumerate}
\end{defn}

Next, we show an example of a simply connected algebra which is not strongly simply connected.

\begin{exmp} Let $Q$ be  the quiver

$$
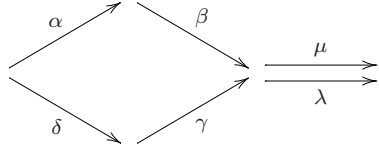
\xymatrix @!0 @R=1.0cm  @C=1.7cm  {
&\ar[rd]^{\beta} & &\\
   \ar[ru]^{\alpha} \ar[rd]_{\delta}  &  & \ar@< 3pt>[r]^{\mu}\ar@<-3pt>[r]_{\lambda}&\\
        & \ar[ru]_{\gamma} & &
    }
$$

\noindent and $I$ be the ideal generated by $\alpha\beta-\gamma\delta$ and $\alpha\beta\lambda-\alpha\beta\mu$.

The algebra $A \simeq kQ/I$ is  triangular. Observe that $\Pi_1(Q,I)=1$.  In fact, since $\alpha\beta \sim \gamma\delta$ then  it is left to prove that $\mu \sim \lambda$.
We know that $\alpha\beta\mu \sim \alpha\beta\lambda$ then $\beta^{-1}\alpha^{-1}\alpha\beta\mu \sim \beta^{-1}\alpha^{-1}\alpha\beta\lambda$ and therefore $\mu \sim \lambda$.

In order to see that $A$ is simply connected we shall consider all possible presentations of $A$ and prove that for each presentation the fundamental group is trivial. In this case, we have another possible presentation as follows;  $\varphi: kQ/I \rightarrow kQ/I'$  with $\varphi (I)=I'$ that maps $\alpha \rightarrow \alpha$, $\beta  \rightarrow \beta$,
$\delta \rightarrow \delta$, $\gamma \rightarrow \gamma$,  $\mu  \rightarrow  \mu+ a\lambda$ with $a \in k$ and $\lambda \rightarrow \lambda$. Let see that $\Pi_1(Q,I')=1$.

With this new presentation, we have that if we write $\mu'= \mu+ a\lambda$ then $\mu' \sim \lambda$.
Indeed,

$$
\begin{array}{lcl}
\alpha\beta\mu'  &=& \alpha\beta(\mu+ a\lambda) \\
&=& \alpha\beta\mu+ a \alpha\beta\lambda\\
&=& (a+1) \alpha\beta\mu\\
&\sim & (a+1) \alpha\beta\lambda.\\
\end{array}$$

\noindent Hence, we conclude that $\mu' \sim \lambda$ and therefore $\Pi_1(Q,I')=1$.
Note that any other possible change of presentation is similar to the one analyzed before.

On the other hand, we observe that the algebra $A$ is not strongly simply connected because there is a presentation of $A \simeq kQ/I'$ that has a full convex subquiver of the following form

$$
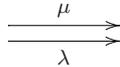
\xymatrix @!0 @R=1.0cm  @C=1.7cm  { &
\ar@< 3pt>[r]^{\mu}\ar@<-3pt>[r]_{\lambda} &\\
        }
$$

 \noindent which is not simply connected.
\end{exmp}


In case we deal with  a finite dimensional triangular $k$-algebra $A$ over an algebraically closed field
$k$, Definition \ref{ss} coincides with the formulation {\it if for any presentation $A \simeq kQ/I$ as a bound
quiver algebra, the fundamental group $\pi_1 (Q,I)$ of $(Q,I)$ is
trivial}, see \cite{BLS}.

\subsection{Galois coverings}

Let $A = kQ/I$ and $A'= kQ'/I'$.
Let $F :A'\rightarrow A$ be a surjective map induced from a (surjective)
quiver map $F_1 :Q \rightarrow Q'$ with $kF_{1}(I') \subset I$. We say that $F$ is a Galois covering defined by the action of
a group of automorphisms $G$ of the bound quiver $(Q', I')$ if
\begin{enumerate}
\item $Fg = F$ for any $g \in  G$;
\item $F_1^{-1}(x) = Gx', F_1^{-1}(\alpha)= G \alpha'$
for vertices $x$ in $Q$  and $x'$ in $Q'$
with $F_{1} x'=x$ and
arrows $\alpha$ in $Q$ and $\alpha'$
in $Q'$ with $F_{1} \alpha'=\alpha$;
\item G acts freely on $A'$, that is, if $g \in G$ fixes a vertex in $Q$ then
$g = 1$.
\end{enumerate}

There is a Galois covering $\widetilde{F} :\widetilde{A} = k\widetilde{Q}/ \widetilde{I} \rightarrow A = kQ/I$
defined by the action of
$\pi_1(Q_A, I )$. For any other Galois covering $F':A'= kQ'/I'
\rightarrow A$, there exists a (Galois covering)
map $\overline{F} :\widetilde{A} \rightarrow A'$
such that $F' \overline{F} = \widetilde{F}$. We illustrate the above information in the following diagram:

$$
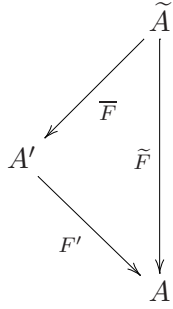
\xymatrix@+1pc{ &\widetilde{A}\ar[dd]_{\widetilde{F}}\ar[ld]^{\overline{F}}& \\
 A' \ar[rd]_{F'}                  & & \\
                   & A      &      }$$\\

\noindent Moreover, the group H of automorphisms of $\widetilde{A}$
defining $\overline{F}$ is normal in $\pi_1(Q_A, I)$ and $\pi_1(Q_A, I)/H $
is the group of automorphisms of $A'$
defining $F'$.
Therefore, $\widetilde{F} :\widetilde{A} \rightarrow A$  is called the universal Galois covering of $A$ with
respect to $(Q, I)$, see \cite{MP}.

\subsection{On naturally
contractible cycles.} The following results (theorem and corollary) were proved  by
I. Assem, S. Liu and Y. Zhang in a private communication. For the benefit of the
reader, we present their proofs here.

\begin{thm}\emph{[Assem, Liu, Zhang]}\label{thmALZ} Let $A$ be a connected
triangular locally bounded $k$-category.
If $A$ is strongly simply connected then
for any presentation $A \cong kQ_A/I_A$ of $A$ we have that every simple cycle of $Q_A$ is naturally
contractible in $(Q_A, I_A)$.\end{thm}
\begin{proof}
Let $A\cong kQ_A/I_A$
be a presentation of $A$. By \cite[Theorem 1.3]{AL}, every irreducible
cycle of $Q_A$ is an
irreducible contour that is naturally contractible in $(Q_A, I_A)$. To show that every contour
in $Q_A$ is naturally contractible in $(Q_A, I_A)$, we shall recall the partial order on
contours defined in \cite{AL} as follows. If $(p_i, q_i)$ is a contour from $a_i$ to $b_i,\; i=1, 2,$ then
$(p_2, q_2)\le (p_1, q_1)$ provided that  $(p_1, q_1)=(p_2, q_2)$ or otherwise
$(a_1, b_1)\ne (a_2, b_2)$ and $a_1$ is a predecessor of $a_2$ while $b_1$ is a successor of $b_2$.
Let $(p, q)$ be a contour in $Q_A$. If $(p, q)$ is minimal, then
$(p, q)$ is irreducible, and hence naturally contractible. Suppose that $(p, q)$ is not minimal and
every contour less than $(p, q)$ is naturally contractible. If $(p, q)$ is irreducible, then it is
naturally contractible. Otherwise there exists a sequence of paths $p=p_0, p_1, \ldots, p_n=q$ from
$s(p)$ to $e(p)$ such that $(p_i, p_{i+1})$ is a interlaced contour for all $0\le i<n.$ Write
$p_i=u_1u_2$ and $q_i=v_1v_2$ with $u_1, u_2, v_1$ and $v_2$ non-trivial paths such that
$e(u_1)=e(v_1)$. Then $(u_1, u_2)$ and $(u_2, v_2)$ are two contours less than $(p_i, q_i)$.
Therefore $(u_1, u_2)$ and $(u_2, v_2)$ are naturally contractible, and hence so is
$(p_i, p_{i+1})$. This implies that $(p, q)$ is naturally contractible.

Let now $C$ be a simple cycle of $Q$. If $\sigma(C)=1$, then $C$ forms a contour $(p, q)$. Hence
$C$ is naturally contractible by our claim.  Suppose now that $\sigma(C)>1$ and all simples
cycles $C'$ with $\sigma(C')<\sigma(C)$ are naturally contractible. Note that $C$ is reducible since
$C$ is not a contour. Therefore, there exists a path $u$ from a point $x\in C$ to another point
$y\in C$ such that $w_1p^{-1}, \; pw_2^{-1}$ are simple cycles with $\sigma(w_1p^{-1})< \sigma(C)$
and $\sigma(w_2p^{-1})<\sigma(C)$, where $w_1, w_2$ are the  reduced walks from $x$ to $y$ such that
$C=w_1w_2^{-1}$. By the inductive assumption, the  simple cycles $w_1p^{-1}, \; pw_2^{-1}$ are
naturally contractible. Therefore $C$ is naturally contractible  by definition. This completes the
proof. \end{proof}

\begin{cor}\emph{[Assem, Liu, Zhang]} \label{coroALZ} Let $A$ be a strongly simply
connected locally bounded $k$-category.
Then for any presentation $A=kQ_A/I_A$, every reduced cycle of $Q_A$
is naturally contractible in
$(Q_A, I_A)$.\end{cor}
\begin{proof}
Let $n$ be the minimal length of reduced cycles of $Q$. Let $C$ be
a reduced cycle. If $C$ is of length $n$, then $C$ is a simple
cycle. Hence $C$ is naturally contractible by Theorem
\ref{thmALZ}. Assume that the length of $C$ is greater than $n$.
If $C$ is a simple cycle, then it is naturally contractible by
Theorem \ref{thmALZ}. Otherwise we may assume that $C=C_1C_2$,
where $C_i$ are reduced cycles shorter than $C$. So each $C_i$
is naturally contractible by the inductive hypothesis. Hence $C$
is naturally contractible by Lemma \ref{Lemma 1.6}. \end{proof}

\subsection{ Algebras of the first kind.}
Let $A$ be a basic connected finite dimensional $k$-algebra with
unit, over an algebraically closed field $k$. Consider an
epimorphism of $k$-algebras $\nu:kQ_A \rightarrow A$  such that
$A\simeq kQ_A /\mbox{ker} \, \nu$. Let $(Q_A, I_\nu)$ be a
presentation of $A$. For any A-module $X$ we consider the convex
full subcategory $A(X)$ of $A$ induced by those vertices in
$\mbox{supp}\,X$. There is an induced presentation $\nu_{_X}:kQ_X
\rightarrow A(X)$ given by the restriction of $\nu$. An indecomposable $A$-module $X$ is said to be
of the first kind with
respect to $\nu$ if there is a $\widetilde{B}$-module $Y$ such
that $F'_{\lambda}Y=X$ for the push-down functor $F'_{\lambda}$
associated to the universal Galois covering $F : \widetilde{B}
\rightarrow B = A(X)$ with respect to $\nu_{_X}$. Furthermore, an algebra $A$
is said to be of the first kind with respect to $\nu_{_X}$ if for
every vertex $x$ in $Q_A$ and every indecomposable direct summand
$X$ of $\mbox{rad}P_{X}$, the module $X$ is of the first kind.
We observe that any representation-finite triangular algebra is of the first kind.

If $A$ is a triangular algebra, that is, $A=kQ_A/\mbox{ker} \, \nu$
where $Q_A$ has no oriented cycles then a  vertex $x$ in $Q_A$ is said to be
separating if for the indecomposable decomposition $\mbox{rad}P_x = M_1 \oplus \dots \oplus M_s$, there is
a decomposition into connected components $Q_A^{(x)}= Q_1 \coprod \dots \coprod Q_t$,
 with $t\geq  s$, of the
induced full subquiver $Q_A^{(x)}$ of $Q_A$ with vertices those $y$ which are not predecessors
of $x$, such that $\mbox{supp} M_i = \{y \in Q_A: M_i(y) = 0\} \subset Q_i$, for $1 \leq   i  \leq s$.

Following \cite{BLS}, the algebra $A$ is separated if every vertex $x$ in $Q_A$ is separating.

We recall that a $k$-category $A$ is schurian if $dim_k A(x, y) \leq 1$ for every couple of
vertices $x,y$ in $Q_A$.\\

Finally, we state the main result given by the authors in \cite{CP}, since it will be useful for our further considerations.\\

\noindent{\bf Theorem}(\cite[Theorem 2.3]{CP}) Let $A$ be a triangular algebra of the first kind with respect to the presentation $\upsilon: kQ_A \rightarrow A$. Consider the
universal covering $F :\widetilde{A}_\upsilon \rightarrow A$. Then, the following conditions hold.
\begin{enumerate}
\item The fundamental group $\pi_1(Q_A, I_\upsilon)$ is free.
\item If $A$ is schurian then $\widetilde{A}_\upsilon$ is separated.
\end{enumerate}

For unexplain notions in this subsection we refer the reader to \cite{CP}.

\section{The main Result}

The aim of this section is to prove our main result. We shall prove
that given a triangular algebra $A$ which admits a strongly simply
connected Galois covering for a given presentation of $A$ then
the fundamental group associated is a free group.

We start this section proving some preliminaries results in order to
get the main result.

\begin{lem} \label{lema1} Let $Q$ be a locally finite quiver and $\gamma$ be a
reduced cycle in $Q$. Assume that there exist two vertices $a, b$
in $\gamma$, a path $p$ from $a$ to $b$ and walks $w_1, w_2$ from
$a$ to $b$ such that $\gamma =w_1 w_2 ^{-1}$, where $w_1 p
^{-1}$ and $w_2 p ^{-1}$ are reduced cycles. Then,
exactly one of the following statements hold.
\begin{enumerate}
\item[(a)] If $a$ is a sink and $b$ is a source in $\gamma$ then $\sigma (w_1 p^{-1}) + \sigma (w_2 p^{-1})= \sigma(\gamma)-1$.

\item[(b)] If $a$ is not a sink and $b$ is
a source in $\gamma$ or either if $a$ is a sink and $b$ is not a source in
$\gamma$ then $\sigma (w_1 p^{-1}) + \sigma (w_2 p^{-1})=
\sigma(\gamma)$.
\item[(c)] If neither $a$ is a sink nor $b$ is
a source in $\gamma$ then $\sigma (w_1 p^{-1}) + \sigma (w_2 p^{-1})=
\sigma(\gamma)+1$.
\end{enumerate} \end{lem}

\begin{proof} $(a)$  Suppose that $a$ is a sink and $b$ is a source in
$\gamma$. Then, we have the following situation:

$$
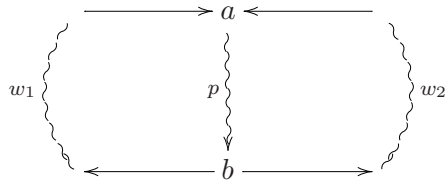
\xymatrix@+2pc{
 \ar[r] \ar@/_1pc/@{~}[d]_{w_1} & a \ar@{~>}[d]_p                   & \ar[l] \ar@/^1pc/@{~}[d]^{w_2} \\
                   & b \ar[r]  \ar[l]           &  }$$\\

\noindent Any source in $\gamma$ which is different from $b$
is a source in $w_1 p^{-1}$ or a source in $w_2 p^{-1}.$ Clearly,
the sources in $w_1 p^{-1}$ and also in $w_2 p^{-1}$ are sources in
$\gamma$. Thus, $\sigma(w_1 p^{-1})+ \sigma(w_2 p^{-1})=
\sigma(\gamma) -1$.

$(b)$ Suppose that $a$ is not a sink and that $b$ is a source in
$\gamma$. Then, the following situation holds:

$$
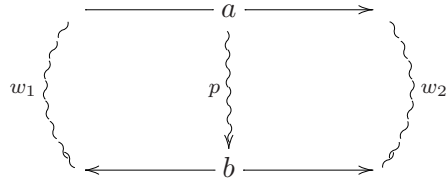
\xymatrix@+2pc{
 \ar@{-}[r] \ar@/_1pc/@{~}[d]_{w_1} & a \ar@{~>}[d]_p    \ar[r]               &  \ar@/^1pc/@{~}[d]^{w_2} \\
                   & b \ar[r]  \ar[l]           &  }$$\\

\noindent The sources in $\gamma$ which are different from $b$ are either
sources in $w_1 p^{-1}$ or in $w_2 p^{-1}.$  The vertex $b$ is no
longer a source in $w_1 p^{-1}$  neither in $w_2 p^{-1}$. But, $a$
is a source of one of the cycles $w_1 p^{-1}$  or $w_2 p^{-1}$.
Moreover, $a$ is a source in  both mentioned cycles if $a$ is a
source in $\gamma$. Clearly, all sources in $w_1 p^{-1}$ and all
sources in $w_2 p^{-1}$  which are different from $a$ are sources
of $\gamma$. Thus, $\sigma(w_1 p^{-1})+ \sigma(w_2 p^{-1})=
\sigma(\gamma)$.

Now, suppose that $a$ is a sink and that $b$ is not a source in
$\gamma$. Then, we may assume that we have the following diagram:

$$
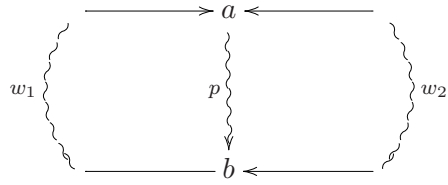
\xymatrix@+2pc{
 \ar[r] \ar@/_1pc/@{~}[d]_{w_1} & a \ar@{~>}[d]_p                   & \ar[l] \ar@/^1pc/@{~}[d]^{w_2} \\
                   & b   \ar@{-}[l]           & \ar[l] }$$\\

\noindent The sources in $\gamma$ coincide with the sources in
$w_1 p^{-1}$ and also with the sources in $w_2 p^{-1}.$ Hence, $\sigma(w_1 p^{-1})+
\sigma(w_2 p^{-1})= \sigma(\gamma)$.

$(c)$ Suppose that $a$ is not a sink and $b$ is not a source
in $\gamma$. Then, we have one of the following situations:

$$
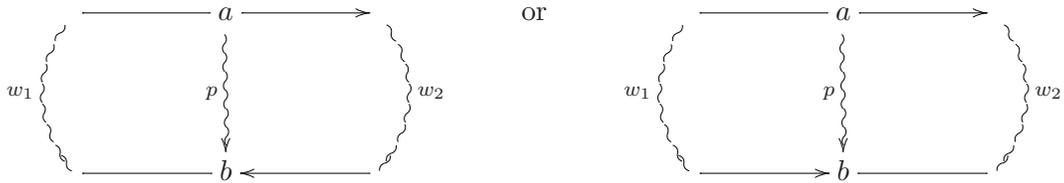
\xymatrix@+2pc{
 \ar@{-}[r] \ar@/_1pc/@{~}[d]_{w_1} & a \ar@{~>}[d]_p   \ar[r]   & \ar@/^1pc/@{~}[d]^{w_2} &\text{or}&
  \ar@{-}[r] \ar@/_1pc/@{~}[d]_{w_1} & a \ar@{~>}[d]_p  \ar[r]   &  \ar@/^1pc/@{~}[d]^{w_2}\\
                   & b  \ar@{-}[l]           & \ar[l] &&
     \ar[r]              & b \ar@{-}[r]          &  }$$\\

\noindent The sources different from $a$ in $w_1 p^{-1}$ and in $w_2 p^{-1}$
are sources of $\gamma$. The vertex $a$ is a
source in one of the cycles $w_1 p^{-1}$ or $w_2 p^{-1}$ (both if
$a$ is a source in $\gamma$). Then, $\sigma(w_1 p^{-1})+
\sigma(w_2 p^{-1})= \sigma(\gamma)+1$.
\end{proof}

\begin{prop}  \label{prop2} Let $kQ_A/I_A$ be a presentation of a triangular
algebra $A$ and $F:(\widetilde{Q}_A, \widetilde{I}_A) \rightarrow
(Q_A, I_A)$ be a Galois covering of bound quivers. Let  $\gamma$
be a simple cycle in $\widetilde{Q}_A$ of the form

$$
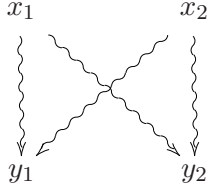
\xymatrix@+2pc{
 x_1 \ar@{~>}[d] \ar@{~>}[rd] & x_2 \ar@{~>}[d] \ar@{~>}[ld] \\
 y_1                  & y_2 }$$\\

\noindent where $x_i \rightsquigarrow y_j$  are paths for $1 \leq
i,j \leq 2$, $F(x_1)= F(x_2)=x$ and  $F(y_1)= F(y_2)=y$. Then,
$\gamma$ is irreducible. Consequently, $\gamma$ is not naturally
contractible in $(\widetilde{Q}_A, \widetilde{I}_A)$.
\end{prop}

\begin{proof} Assume that $\gamma$ is reducible. Then, there exist two
points $a, b$ in $\gamma$, a path $p:a \rightsquigarrow b$ and walks
$w_1$ and $w_2$ from $a$ to $b$  such that
$\gamma = w_1 w_2^{-1}$, $\sigma(w_1 p^{-1}) < \sigma(\gamma) =2$
and $\sigma(w_2 p^{-1})< \sigma(\gamma) =2$.

Since $A$ is triangular, the cycles $w_1 p^{-1}$ and $w_2 p^{-1}$ are contours
and we have that $\sigma(w_1 p^{-1}) + \sigma(w_2 p^{-1}) = \sigma(\gamma)$.
By Lemma \ref{lema1}, we have that either $b$ is a source or $a$ is a sink.
Therefore, $F(b)=x$ or $F(a)=y$. Since for all vertex $z$ of $\gamma$ we have
a path in $Q_A$ from $x$ to $y$ going through $F(z)$,  then we claim that in both cases, the path
$F(p)$ lies in a cycle. In fact, we have $x \rightsquigarrow F(a) \rightsquigarrow F(b) = x$
or $y=F(a) \rightsquigarrow F(b) \rightsquigarrow y$. This yields the desired contradiction,
since $A$ is triangular. Hence, $\gamma$ is irreducible.
\end{proof}

We define  the distance between two points $x, y \in Q_A$ as follows:

$$d(x,y)= max \{ \ell(w) \mid w:x \rightsquigarrow y \;\;\mbox{is
a path} \}.$$

\noindent where $\ell(w)$ means the length of the path $w$. Note that the distance between two points in $Q_A$ is well defined since $A$ is triangular.\vspace{.1in}

Now we prove a technical lemma useful for further considerations.

\begin{lem}  \label{lema3} Let $kQ_A/I_A$ be a presentation of a triangular algebra $A$ and
$F:(\widetilde{Q}_A, \widetilde{I}_A) \rightarrow (Q_A, I_A)$ be a
Galois covering of bound quivers. Let $(u,v)$, $(u, w)$ and $(v,w)$ be torsion contours
from $x$ to $y$ in $Q_A$ of minimum distance $d(x,y)$. Let
$\gamma$ be a simple cycle in $\widetilde{Q}_A$ of the form

$$
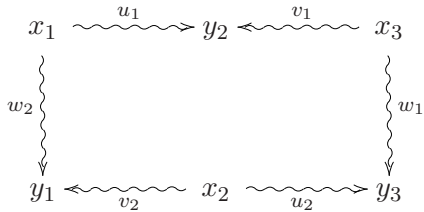
\xymatrix@+2pc{
 x_1 \ar@{~>}[d]_{w_2} \ar@{~>}[r]^{u_1} & y_2 & x_3 \ar@{~>}[d]^{w_1} \ar@{~>}[l]_{v_1} \\
 y_1 & x_2 \ar@{~>}[r]_{u_2} \ar@{~>}[l]^{v_2} & y_3 }$$\\

\noindent where $u_1, u_2, v_1, v_2, w_1$ and $w_2$ are paths,
$F(x_i)=x,$  $F(y_i)= y$ for $i=1, 2, 3$, $F(u_1)= F(u_2)=u$,
$F(v_1)= F(v_2)=v$ and $F(w_1)= F(w_2)=w$.
Then, $\gamma$ is not naturally contractible in
$(\widetilde{Q}_A, \widetilde{I}_A)$.
\end{lem}

\begin{proof}
Assume that $\gamma$ is naturally contractible. Then, there exist
vertices $a, b$ in $\gamma$, a path $p$ and
walks $\delta_1$ and $\delta_2$ all of them from $a$ to $b$  such that $\gamma
= \delta_1 \delta_2^{-1}$, $\sigma(\delta_1 p^{-1}) <
\sigma(\gamma) =3$ and $\sigma(\delta_2 p^{-1})< \sigma(\gamma)
=3$. Moreover, $\delta_1 p^{-1}$ and $\delta_2 p^{-1}$ are
naturally contractible reduced cycles. Since for all vertex $z$ in
$\gamma$ there is a path from $x$ to $y$ in $Q_A$ going through
$F(z)$ and $A$ is triangular then we may assume that
$F(a)\neq y$ and  $F(b)\neq x$. Otherwise, if $F(a)=y$ or $F(b)=x$, the path $F(p)$ lies in a cycle getting a contradiction.

Then, $a$ is not a sink in $\gamma$ and $b$ is not a source in $\gamma$.
Since $A$ is triangular, $\widetilde{Q}_A$ has no oriented cycles and $\sigma(\delta_i p^{-1})\geq 1$, for $i=1,2$.
Then by Lemma \ref{lema1}, we have that
$\sigma(\delta_1 p^{-1}) + \sigma(\delta_2 p^{-1}) = \sigma({\gamma})+1 = 4$.
Therefore, it follows  from the fact that $\sigma(\delta_i p^{-1}) < 3$, for $i=1,2$
that $\sigma(\delta_i p^{-1}) = 2$, for $i=1,2$.

Without loss of generality we may assume that $a$ is a vertex
of $u_1$. It is not hard to see that either $b$ is in $u_2$ or $b$
in $w_1.$  Indeed if $b \in v_1$ then there exists a cyclic permutation of $\delta_1 p^{-1}$ which is a contour contradicting the number of sources. A similar analysis holds if $b \in v_2$.

Now, if $b$ is in $w_1$ then we have a diagram as follows

$$
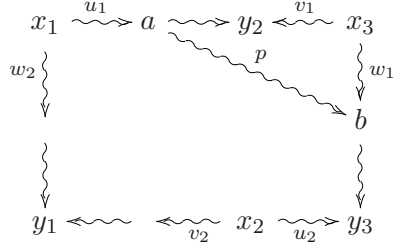
\xymatrix{
 x_1 \ar@{~>}[d]_{w_2} \ar@{~>}[r]^{u_1} & a \ar@{~>}[rrd]^{p} \ar@{~>}[r]& y_2 & x_3 \ar@{~>}[d]^{w_1} \ar@{~>}[l]_{v_1} \\
  \ar@{~>}[d] & &  &  b \ar@{~>}[d] \\
 y_1&\ar@{~>}[l] & x_2 \ar@{~>}[r]_{u_2} \ar@{~>}[l]^{v_2}& y_3 }$$

\noindent then either one of the naturally
contractible cycles $\delta_1 p^{-1}$ or $\delta_2 p^{-1}$ is a  simple
cycle of the form:

$$
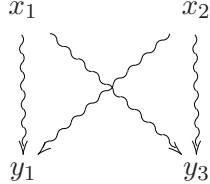
\xymatrix@+2pc{
 x_1 \ar@{~>}[d] \ar@{~>}[rd] & x_2 \ar@{~>}[d] \ar@{~>}[ld] \\
 y_1                  & y_3 }$$\\

\noindent Then, Proposition \ref{prop2} yields the desired contradiction.

Now, if $b$ is in $u_2$, then Proposition \ref{prop2} implies that $F(a), F(b)\notin \{x, y\}$. Since $A$ is triangular,
there exist paths $q_1: x \rightarrow F(a)$, $q_2: F(a) \rightarrow F(b)$ and  $q_3: F(b) \rightarrow y$ such that $u=q_1 q_2 q_3$.

We claim that $(F(p), q_2)$ is a torsion contour, where $F(p):F(a) \rightarrow F(b)$ is a path in $Q_A$. Observe that in $\widetilde{Q}_A$,
for $i=1,2,3$ there exist paths $q'_i$ and $q''_i$ such that
$F(q'_i)=F(q''_i)=q_i,$ $u_1=q'_1 q'_2 q'_3$
and  $u_2=q''_1 q''_2 q''_3.$ Hence,  we may assume
that $\delta_1= (q'_1)^{-1} w_2 v_2^{-1}q''_1 q''_2$ and $\delta_2= q'_2 q'_3 (v'_1)^{-1} w_1 {q''_{3}}^{-1}.$
By definition
$\delta_i p^{-1} \sim 1$ for $i=1,2$,  since they are naturally contractible reduced cycles. Thus,
$F(\delta_i) F( p^{-1}) \sim 1$ for $i=1,2$. Therefore,
$$q_2^{-1}F( p) \sim q_2^{-1} F(\delta_2)= q_2^{-1}q_2 q_3 v^{-1} w q_3^{-1} \sim  q_3 v^{-1} w q_3^{-1}.$$

Since $(v,w)$ is a torsion contour then $v^{-1} w \sim 1$. Hence, we infer that $(F(p),q_2)$ is
also a torsion contour because $(F(p),q_2)$  is a contour such that $(q_2^{-1}F( p))^{2} \sim 1$ getting a contradiction to the minimality of $(v,w)$. Therefore,
$\gamma$ is not naturally contractible in $(\widetilde{Q}_A, \widetilde{I}_A)$.
\end{proof}

\begin{lem}  \label{lema4} Let $kQ_A/I_A$ be a presentation of a triangular algebra $A$ and
$F:(\widetilde{Q}_A, \widetilde{I}_A) \rightarrow (Q_A, I_A)$ be a
Galois covering of bound quivers with $\widetilde{A} \simeq
k\widetilde{Q}_A/\widetilde{I}_A$ strongly simply connected. Let $(u,v)$ be a torsion contour
from $x$ to $y$ in $Q_A$ of minimum distance $d(x,y)$. Then, there
exists a simple cycle $\gamma$ in $\widetilde{Q}_A$ of the form
$$
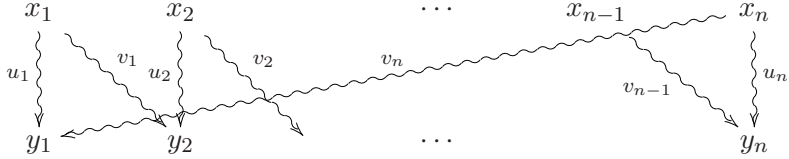
\xymatrix@+1pc{
 x_1 \ar@{~>}[d]_{u_1} \ar@{~>}[rd]^{v_1} & x_2 \ar@{~>}[d]_{u_2} \ar@{~>}[rd]^{v_2} & &\cdots& x_{n-1} \ar@{~>}[rd]_{v_{n-1}} & x_n \ar@{~>}[d]^{u_n} \ar@{~>}[llllld]_{v_n}\\
 y_1                  & y_2 & &\cdots& & y_n}$$\\

\noindent where $n$ is the order of the contour $(u,v)$, $u_i:x_i
\rightsquigarrow y_i$ and  $v_i:x_i \rightsquigarrow y_{i+1}$ are
paths in $\gamma$  with $y_{n+1}=y_1$, $F(u_i) = u$,
$F(v_i) = v$,
$F(x_i)=x$ and  $F(y_i)= y$ for $i=1, \dots, n$.
\end{lem}

\begin{proof} Since $(u,v)$ is a contour of $Q_A$ of order $n$ then $(u v ^{-1})^{n} \sim 1.$
Moreover, there exists a reduced cycle $\gamma$ in
$\widetilde{Q}_A$ of the desired form. It remains to
prove that $\gamma$ is simple. Suppose that $\gamma$ is not
simple. Then, there exists a vertex $a$ which appears in $\gamma$
at least two times. We analyze all the possible cases where
such a vertex can appear in $\gamma$. \vspace{.05in}

{\it Case 1:} Consider $a=x_i=x_j$ with $i \neq j$. Without loss of generality we may assume that $i<j$. Then, in
$\widetilde{Q}_A$ we have the following situation:

$$
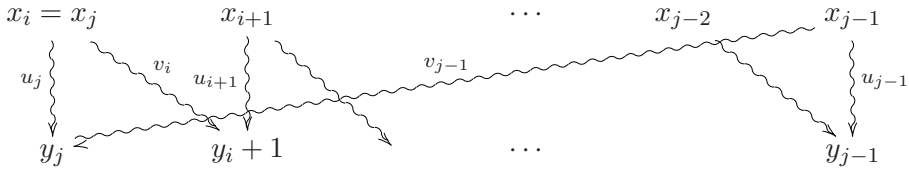
\xymatrix@+1pc{
 x_i = x_j \ar@{~>}[d]_{u_j} \ar@{~>}[rd]^{v_i} & x_{i+1} \ar@{~>}[d]_{u_{i+1}} \ar@{~>}[rd] & &\cdots&x_{j-2} \ar@{~>}[rd] & x_{j-1} \ar@{~>}[d]^{u_{j-1}} \ar@{~>}[llllld]_{v_{j-1}}\\
 y_j                  & y_i+1 & &\cdots& & y_{j-1}}$$\\


Hence $(u v ^{-1})^{j-i} \sim 1,$ a contradiction to the fact that the order of $(u,v)$ is $n$ and $i<j<n$. \vspace{.05in}

{\it Case 2:} Consider $a=x_{i}$ for exactly one $i$. Since $a$ appears at least two times in $\gamma$,
then there exist $j \neq i$ and a path from $x_j$ to $a$ in
$\widetilde{Q}_A.$

Note that such a path induces a cycle in $Q_A$ since $F(a)=
F(x_i)=F(x_j)=x$. This fact leads to a contradiction since
$A$ is triangular.

To analyze the others cases we can assume that $F(a) \neq x$. Similarly, we can also assume that
$F(a) \neq y$. \vspace{.05in}

{\it Case 3:} Consider $a \in u_i \cap u_j$ (or $a \in v_i \cap v_j$) with $i \neq j$.
Since $F(a) \neq x$ there exist two arrows $\alpha: b \rightarrow
a$ and $\beta: c \rightarrow a$ such that $u_i= w_i\alpha w'_i$
and $u_j= w_j \beta w'_j.$  We may assume that $b \neq c$. Since
$F(u_i)= F(u_j)=u$  where $u$ is a simple path and
$e(F(\alpha))= F(a)= e(F(\beta))$ we have that
$F(\alpha)=F(\beta)$.

On the other hand, $e(\alpha)=a=e(\beta)$. Thus, $\alpha=\beta$ a contradiction to our assumption. \vspace{.05in}

{\it Case 4:} Consider $a \in u_i \cap v_j$. Then, there exists a path
$p_1:x \rightsquigarrow F(a)$, $p_2:F(a) \rightsquigarrow y$ and $q_1:x
\rightsquigarrow F(a)$, $q_2:F(a) \rightsquigarrow y$ in $Q_A$ such that
$u=p_1p_2$ and $v=q_1q_2.$

Without loss of generality, we may assume that $i=1$. Then, in $\widetilde{Q}_A$
we have the following situation:

$$
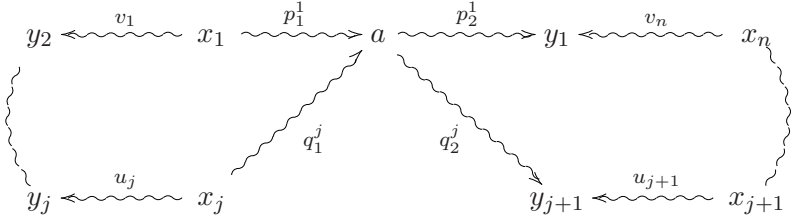
\xymatrix@+2pc{
 y_2 \ar@/_1pc/@{~}[d] & x_1 \ar@{~>}[l]_{v_1} \ar@{~>}[r]^{p_1^1} & a \ar@{~>}[r]^{p_2^1} \ar@{~>}[rd]_{q_2^j} & y_1 & x_n \ar@{~>}[l]_{v_n} \ar@/^1pc/@{~}[d] \\
 y_j & x_j \ar@{~>}[l]_{u_j} \ar@{~>}[ru]_{q_1^j} &   & y_{j+1} & x_{j+1} \ar@{~>}[l]_{u_{j+1}} }$$\\


\noindent where $p_1^{1}, \;p_2^{1}, \;q_1^{j}$ and $ q_2^{j}$ are
paths.

Then, $q_1^{-1}(uv^{-1})^{j-1}p_1 \sim 1$ and
$p_2(uv^{-1})^{n-j}q_2^{-1} \sim 1$ since $\widetilde{A}$ is strongly simply connected. Thus, $(uv^{-1})^{j-1} \sim
q_1 p_1^{-1}$ and $(uv^{-1})^{n-j} \sim p_2^{-1}q_2.$

Observe that $1< j < n$. Thus, $2\leq j \leq n-1$. Since $(u,v)$ is a
contour in $Q_A$ of order $n$, then we have that
$(p_1,q_1)$ is a torsion contour contradicting the
minimality of the distance of $(u,v)$. Hence, we prove that
$\gamma$ is a simple cycle in $\widetilde{Q}_A$.
\end{proof}

\begin{prop}  \label{prop5} Let $kQ_A/I_A$ be a presentation of a triangular algebra $A$ and
$F:(\widetilde{Q}_A, \widetilde{I}_A) \rightarrow (Q_A, I_A)$ be a
Galois covering of bound quivers with $\widetilde{A} \simeq
k\widetilde{Q}_A/\widetilde{I}_A$ strongly simply connected. If
$(u,v)$ is a contour in $Q_A$ then $(u,v)$ is not of torsion.
\end{prop}
\begin{proof} Assume that  there exists $(u',v')$ a torsion contour in $Q_A$.
Consider the set $$S=\{(u',v')\mid (u',v') \text{ is a torsion contour in }Q_A \}.$$
Let $(u,v) \in S$ be a contour in $Q_A$ from $x$ to $y$ with minimal distance $d(x,y)$, that is, an element of $S$ starting and ending in points such that their distance is minimal related to all the starting and ending points of the contours in $S$.

We claim that  $(u,v)$ is reduced. In fact, otherwise there exist paths $\alpha, \beta,  u'$ and $v'$ in $Q_A$
such that $u=\alpha u' \beta$ and $v=\alpha v' \beta.$

Note that either $\alpha$ or $\beta$ can be trivial paths but not both. Therefore,
$$uv^{-1} \sim \alpha u' (v')^{-1} \alpha^{-1}$$where $(u',v')$ is a torsion contour.
This fact contradicts the minimality of $(u,v)$. Thus the contour
$(u,v)$ is reduced.

Consider $\gamma =u v^{-1}$  starting in $x$ and suppose that the order of $(u,v)$ is $n$ with $n >1$.
By Lemma \ref{lema4}, we have a simple cycle $\widetilde{\gamma}$ in $\widetilde{Q}_A$ of the form:

$$
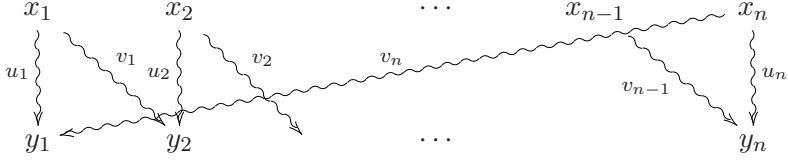
\xymatrix@+1pc{
 x_1 \ar@{~>}[d]_{u_1} \ar@{~>}[rd]^{v_1} & x_2 \ar@{~>}[d]_{u_2} \ar@{~>}[rd]^{v_2} & &\cdots& x_{n-1}\ar@{~>}[rd]_{v_{n-1}} & x_n \ar@{~>}[d]^{u_n} \ar@{~>}[llllld]_{v_n}\\
 y_1                  & y_2 & &\cdots& & y_n}$$\\

On the other hand, since $\widetilde{A}$ is strongly simply connected, by Corollary \ref{coroALZ}
we know that $\widetilde{\gamma}$ is naturally contractible in $(\widetilde{Q}_A, \widetilde{I}_A).$

Since, $\sigma(\widetilde{\gamma})=n >1$ then there exist a path $p:a \rightsquigarrow b$, and two walks
$w_1$ and $w_2$ from $a$ to $b$ such that $\widetilde{\gamma}= w_1 w_2 ^{-1}$ and $w_1 p ^{-1}$,
$w_2 p ^{-1}$ are naturally contractible reduced cycles with $\sigma(w_1 p ^{-1}) <n $  and $\sigma(w_2 p ^{-1}) <n.$

Therefore, there exist walks $\gamma_1, \gamma_2$ in $Q_A$ from
$F(a)$ to $F(b)$ such that $\gamma_{_{F_{a}}} = \gamma_1 \gamma_2^{-1}.$ Observe that the cycle $\gamma_{_{F_{a}}}$ is a cycle starting in $F_{a}$, but the underlying graphs of both cycles $\gamma$ and $\gamma_{_{F_{a}}}$ are the same.
Then, $F(w_1) = \gamma_{_{F_{a}}}^{r}\gamma_1$ and $F(w_2) =
\gamma_{_{F_{a}}}^{-s}\gamma_2$ for some positive integer $r$ and $s$. Since,
$\gamma_{_{F_{a}}}^{n}= F(w_1 w_2^{-1}) =
\gamma_{_{F_{a}}}^{r}\gamma_1\gamma_2^{-1}\gamma_{_{F_{a}}}^{s}= \gamma_{_{F_{a}}}^{r+s+1}$ we
infer that $r+s+1=n$.

Now, since by definition $F(p) \sim F(w_1)= \gamma_{F_{a}}^{r}\gamma_1$ and $F(p) \sim F(w_2)= \gamma_{_{F_{a}}}^{-s}\gamma_2$ then
$F(p) \gamma_1^{-1} \sim  \gamma_{_{F_{a}}}^{r}$ and $F(p) \gamma_2^{-1} \sim  \gamma_{_{F_{a}}}^{-s}= \gamma_{_{F_{a}}}^{r+1-n} \sim \gamma_{_{F_{a}}}^{r+1}.$

Next, we analyze all the possible cases  where $F(a)$ and $F(b)$ may appear in $(u,v)$. \vspace{.05in}

{\it Case 1:} Assume that $F(a)=x$ and $F(b)=y.$ Without loss of
generality we may assume that $a=x_1$ and $b=y_{i+1}$ for some $i
\geq 1$ and,  moreover that $\gamma_1=u$ and
$\gamma_2=v$. Furthermore, $r \neq 0$ and $s \neq 0$ since otherwise
$\sigma(w_1 p^{-1})=1$ or $\sigma(w_2 p^{-1})=1$, a contradiction.

Therefore, $(F(p),u)$  and $(F(p),v)$ are torsion contours in
$Q_A$.  In fact, $F(p)u^{-1} = F(p)\gamma_1^{-1} \sim \gamma^{r}.$ Since $\gamma$ coincides with $\gamma_{F_{a}}$ and $\gamma^{r}$ is of torsion then $(F(p),u)$ is a torsion contour. With a similar analysis we can get that $(F(p),v)$ is a torsion contour.

Hence in  $\widetilde{Q}_A$ we have the following reduced cycle $C$

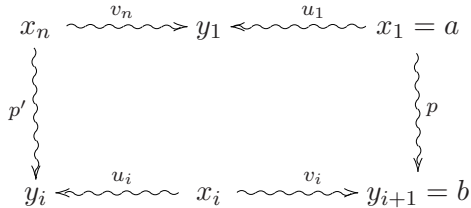
$$\xymatrix@+2pc{
 x_n \ar@{~>}[d]_{p'} \ar@{~>}[r]^{v_n} & y_1 & x_1=a \ar@{~>}[d]^{p} \ar@{~>}[l]_{u_1} \\
 y_i & x_i \ar@{~>}[l]_{u_i} \ar@{~>}[r]^{v_i} & y_{i+1}=b }$$\\

\noindent where $p'$ is obtained shifting $p$ by the respective group element.

By Corollary \ref{coroALZ}, we have that $\mathcal{C}$ is
naturally contractible in $(\widetilde{Q}_A, \widetilde{I}_A),$ a
contradiction to Lemma \ref{lema3}. Thus, $F(a) \neq x$ or
$F(b)\neq y.$ \vspace{.05in}

{\it Case 2:} Consider $F(a) \neq x$ or  $F(b)\neq y.$  We may assume
that $a$ is a vertex in $u_1$. Then, we have to consider two cases
if either $b$ is a vertex in $u_i$ for some $i \neq 1$ or $b$ is a
vertex in $v_i$ for some $i.$

Let first analyze when $b$ is a vertex in $u_i$ for some
$i \neq 1$, otherwise if $i=1$ then $\sigma(w_2 p)=n$ a contradiction.

Since  $F(p)$ is a path in $Q_A$ from $F(a)$ to $F(b)$
and $A$ is a triangular algebra  then there are paths $\delta_1: x
\rightarrow F(a),$  $\delta_2: F(a) \rightarrow F(b)$ and $\delta_3: F(b)
\rightarrow y$ such that $u=\delta_1\delta_2\delta_3$.

We may assume that $\gamma_1=\delta_2$ and $\gamma_2 = \delta_1^{-1} v \delta_3^{-1}.$
Hence, $$F(p) \delta_2^{-1}= F(p) \gamma_1^{-1} \sim
\gamma_{a}^{r}.$$

\noindent  This implies that $(F(p),\delta_2)$ is a  torsion contour in $Q_A$,
getting a contradiction with the minimality of $(u,v).$

On the other hand, if $b$ is a vertex in $v_i$ for some $i$, then
there are paths  $\eta_1: x \rightarrow F(a),$  $\eta_2: F(a)
\rightarrow y$, $\varphi_1: x \rightarrow F(b)$ and  $\varphi_2: F(b)
\rightarrow y$,  such that $u=\eta_1\eta_2$ and $v=\varphi_1\varphi_2$.

We can assume that $\gamma_1 = \eta_1^{-1} \varphi_1$ and  $\gamma_2 = \eta_2
\varphi_2^{-1}.$ We recall that $F(p) \sim \gamma_{a}^r\gamma_1$ and $F(p) \sim \gamma_{a}^{r+1}\gamma_2$,
for some $0\leq r \leq n-1$.
Therefore, replacing $\gamma_1$ by $\eta_1^{-1} \varphi_1$ we get that

$$\eta_1 F(p) \varphi_1^{-1} \sim \eta_1 \gamma_{_{F_{a}}}^{r} \gamma_1 \varphi_1^{-1} = \eta_1 \gamma_{_{F_{a}}}^{r} \eta_1^{-1}$$

\noindent and replacing $\gamma_2$ by $\eta_1^{-1} \varphi_1$ we get that

$$
\begin{array}{lcl}
F(p)\varphi_2 \eta_2^{-1} &\sim & \gamma_{_{F_{a}}}^{r+1} \gamma_2 \varphi_2 \eta_2^{-1}\\
&\sim & \gamma_{_{F_{a}}}^{r+1}.\\
\end{array}$$

We conclude that $(\eta_1 F(p), \varphi_1)$ or $(F(p)\varphi_2, \eta_2)$ are
torsion contours a contradiction to the minimality of $(u, v)$.
\end{proof}

\begin{thm} \label{teo6} Let $A$ be a triangular algebra admitting a
strongly simply connected Galois covering for a given presentation
$kQ_A/I_A.$ Then, the fundamental group of $(Q_A, I_A)$ is
torsion-free.
\end{thm}

\begin{proof} Suppose that $\pi_1(Q_A, I_A)$ is not torsion-free.
Let $\overline{\gamma}$ be a torsion element of the group $\pi_1(Q_A, I_A)$
with $\sigma(\gamma)$ minimum. Observe that $\gamma$ is a cycle in
$Q_A$ starting at the base point $x$ of the group $\pi_1(Q_A, I_A)$ and that $\gamma^n \sim 1$ for some $n \geq 2$.

Consider $\widetilde{A} \simeq k\widetilde{Q}_A/\widetilde{I}_A$ the Galois
covering of the fixed presentation of $A$.  Then, there exists a
cycle $\widetilde{\delta}$ in $\widetilde{Q}_A$ such that
$F(\widetilde{\delta})=\gamma^n$.

Now, if $\widetilde{\gamma}$ is
irreducible then by \cite[Theorem 1.3]{AL} we infer that $\widetilde{\gamma}$
is a contour. Furthermore, $\gamma^n$ is also a contour. Thus,
$n=1$ a contradiction to our assumption.
Then, $\widetilde{\gamma}$ is reducible. Since $\widetilde{\delta}$ is not a
contour there exists in $Q_A$ a path $p:a \rightsquigarrow
b,$ two walks $w_1$ and $w_2$ from $a$ to $b$ such that
$\widetilde{\delta} = w_1 w_2^{-1}$, $\sigma(w_1 p^{-1}) <
\sigma(\widetilde{\delta}) = n \sigma(\gamma)$ and $\sigma(w_2
p^{-1})< \sigma(\widetilde{\delta}) = n \sigma(\gamma)$. Then,
there exist in $Q_A$ walks $\gamma_1$ and $\gamma_2$ from $F(a)$
to $F(b)$ such that $\gamma_{_{F(a)}}= \gamma_1 \gamma_2^{-1}$, where $\gamma_{_{F(a)}}$ is a cyclic permutation of $\gamma$. Observe that
$F(w_1)= \gamma_{_{F(a)}}^r \gamma_1$ and $F(w_2)= \gamma_{_{F(a)}}^{-s} \gamma_2$
where $r+s+1=n$.

By Lemma \ref{lema1}, we have that
$$ \mid \sigma( \gamma_1 F(p^{-1})) + \sigma( \gamma_2 F(p^{-1}))-
\sigma( \gamma_{_{F(a)}}) \mid \leq 1.$$

Note that $\sigma( \gamma_{_{F(a)}})=\sigma( \gamma)$.
Suppose that $ \sigma( \gamma_1 F(p^{-1})) = \sigma( \gamma_{_{F(a)}})$.
Then, $ \mid \sigma( \gamma_2 F(p^{-1})) \mid \leq 1 .$
Since $A$
is triangular we infer that $\mid \sigma( \gamma_2 F(p^{-1})) \mid
= 1$ and therefore $ \gamma_2 F(p^{-1})$ is a contour. Moreover, $
\gamma_2 F(p^{-1}) \sim \gamma_{_{F(a)}}^s$. In fact, since $w_2 p^{-1}$ is contractible due to Corollary \ref{coroALZ}, so $F(p) \sim F(w_2)$ where $F(w_2)= \gamma_{_{F(a)}}^{-s} \gamma_2$. Then $w_2 p^{-1} \sim 1$. Therefore, $\gamma_{_{F(a)}})^s \sim \gamma_2 F(p^{-1})$.

By Proposition \ref{prop5}, we
know that there is no torsion contour. This means that $s=0$ and $r=n-1$.

On the other hand, since $F(w_1 p^{-1})=\gamma^r \gamma_1
F(p^{-1})$ and $\sigma( \gamma_{_{F(a)}})=\sigma( \gamma)$ then we have that

$$
\begin{array}{lcl}
\sigma( w_1 p^{-1}) &=& \sigma( F(w_1 p^{-1})) \\
&=& r \sigma( \gamma)+ \sigma( \gamma_1
F(p^{-1}))\\
&=& (n-1)\sigma( \gamma)+ \sigma( \gamma)\\
&=&n\sigma( \gamma)
\end{array}$$

\noindent a contradiction. Then since  $A$ is triangular, we conclude that
 $\sigma(\gamma_1
F(p^{-1}))< \sigma(\gamma).$ Since $\sigma(\gamma)$ is minimal then
$\gamma_1 F(p^{-1})$ is not of torsion.

On the other hand, since $\gamma_1 F(p^{-1}) \sim \gamma_{_{F(a)}}^{-r}$
then $r=0$ and $s=n-1$. Since $F(w_2 p^{-1})=\gamma_{_{F(a)}}^{-s} \gamma_2
F(p^{-1})$, we have that

$$
\begin{array}{lcl}
\sigma( w_2 p^{-1}) &=& \sigma( F(w_2 p^{-1}))\\
&=& s \sigma( \gamma)+ \sigma( \gamma_2
F(p^{-1}))\\
&=& (n-1)\sigma( \gamma)+ \sigma(\gamma_2 F(p^{-1})).
\end{array}$$

\noindent Since $\sigma( w_2 F(p^{-1})) < n \sigma( \gamma)$ this yields
that $\sigma( \gamma_2 F(p^{-1})) < \sigma( \gamma).$ Again by the
minimality of $\sigma( \gamma)$ we get that $s=0$. Hence, we
prove that $s=0$ and $r=0$. This implies that $n=1$ a
contradiction to our assumption, proving the result.
\end{proof}

\begin{cor} \label{coro7} Let $A$ be a triangular algebra which admits a strongly simply connected
Galois covering for a given presentation $kQ_A/I_A.$ Then, $A$ is
of the first kind.
\end{cor}

\begin{proof} By the above theorem, if $A$ admits a strongly simply connected
Galois covering for a given presentation $kQ_A/I_A$  of $A$, then the fundamental group
$\pi_1(Q_A, I_A)$ is torsion-free. Moreover, by \cite[1.8]{CP}  since
$F_\lambda$ preserves indecomposable modules then $A$ is of the first kind.
\end{proof}

By a result due to Castonguay and de la Pe\~{n}a given in
\cite{CP}, it is known that if $A$ is of the first kind respect to
a given presentation $kQ_A/I_A$  of $A$, then the fundamental
group $\pi_1(Q_A,I_A)$ is free. \vspace{.05in}

Therefore we are in position to state our main theorem.

\begin{thm} \label{teo8} Let $A$ be a triangular algebra which admits a strongly simply connected
Galois covering for a given presentation $kQ_A/I_A$ of $A$. Then,
the fundamental group $\pi_1(Q_A, I_A)$ is free.\end{thm}

\makeatletter
\renewcommand{\@biblabel}[1]{\hfill#1.}
\makeatother

\end{document}